\documentclass[article, 11pt]{amsart}

\usepackage{amsmath, amssymb, amsfonts, amsthm, latexsym}
\usepackage{tikz}

\newtheorem{Theorem}{Theorem}[section]
\newtheorem{Lemma}[Theorem]{Lemma}
\newtheorem{Corollary}[Theorem]{Corollary}
\newtheorem{Proposition}[Theorem]{Proposition}
\newtheorem{Remark}[Theorem]{Remark}
\newtheorem{Example}[Theorem]{Example}

\newcommand{\Q}{\mathbb{Q}}

\newcommand{\Z}{\mathbb{Z}}

\def\depth{\operatorname{depth}}

\def\dim{\operatorname{dim}}

\def\lim{\operatorname{lim}}

\def\mm{{\frak m}}

\def\qq{{\frak q}}

\begin{document}

\title{The second Hilbert coefficient of modules with almost maximal depth}

\author{Van Duc Trung}
\address{Department of Mathematics, University of Genoa, Via Dodecaneso 35, 16146 Genoa, Italy}
\email{vanductrung@dima.unige.it}

\keywords{Local rings, filtration, Hilbert coefficients, associated graded ring }
\subjclass[2010]{Primary: 13H10; Secondary: 13H15}

\begin{abstract}
Let $\mathbb{M} = \{ M_n \}$ be a good $\qq$-filtration of a finitely generated $R$-module $M$ of dimension $d$, where $(R,\mm)$ is a local ring and $\qq$ is an $\mm$-primary ideal of $R$. In case $\depth M \geq d-1$, we give an upper bound for the second Hilbert coefficient $e_2(\mathbb{M})$ generalizing results by Huckaba-Marley and Rossi-Valla proved assuming that $M$ is Cohen-Macaulay. We also give a condition for the equality, which relates to the depth of the associated graded module $gr_{\mathbb{M}}(M)$. A lower bound on $e_2(\mathbb{M})$ is proved generalizing a result by Rees and Narita.
\end{abstract}

\maketitle

\section{Introduction}
Let $(R,\mm)$ be a Noetherian local ring and $M$ a finitely generated $R$-module of dimension $d$. Let $\qq$ be an $\mm$-primary ideal of $R$, we consider $\mathbb{M} = \{ M_n\}$ a $\qq$-filtration of $M$ as follows
$$\mathbb{M}: \quad M = M_0 \supseteq M_1 \supseteq \cdots \supseteq M_n \supseteq M_{n+1} \supseteq \cdots$$
where $M_n$ are submodules of $M$ and $\qq M_n \subseteq M_{n+1}$ for all $n \geq 0$. The $\qq$-filtration $\mathbb{M}$ is called a good $\qq$-filtration if $\qq M_n = M_{n+1}$ for all sufficiently large $n$.

The algebraic and geometric properties of $M$ can be detected by the Hilbert function of a good $\qq$-filtration $\mathbb{M}$, namely $H_{\mathbb{M}}(n) = \lambda(M_n/M_{n+1})$, where $\lambda(-)$ denotes the length as $R$-module. It is well known that there exist the integers $e_i(\mathbb{M})$ for $i=0,1,\cdots,d$ such that for $n \gg 0$
$$H_{\mathbb{M}}(n) = e_0(\mathbb{M})\binom{n+d-1}{d-1} - e_1(\mathbb{M})\binom{n+d-2}{d-2} + \cdots + (-1)^{d-1}e_{d-1}(\mathbb{M}).$$
The integers $e_i(\mathbb{M})$ are called the Hilbert coefficients of $\mathbb{M}$. In particular $e = e_0(\mathbb{M})$ is the multiplicity and it depends only on $M$ and $\qq$.

A rich literature has been produced on the Hilbert coefficients of a filtered module $M$ in the case $M$ is Cohen-Macaulay, for a survey see for instance \cite{RV1}. The first Hilbert coefficient $e_1(\mathbb{M})$ is called Chern number by W.V. Vasconcelos and has been studied very deeply by several authors (see for instance \cite{CZ}, \cite{H}, \cite{HM}, \cite{No}, \cite{P} and \cite{RV2}).

Let $J$ be an ideal of $R$ generated by a maximal $\mathbb{M}$-superficial sequence for $\qq$. If $M$ is Cohen-Macaulay then the following inequalities hold
$$0 \leq e_0(\mathbb{M}) - \lambda(M/M_1) \leq e_1(\mathbb{M}) \leq  \underset{n \geq 0}{\sum} \lambda(M_{n+1}/JM_n)$$
(see \cite{H} and \cite{HM}) and the equalities provide good homological properties of the associated graded module $gr_{\mathbb{M}}(M) = \underset{n \geq 0}{\bigoplus}\ M_n/M_{n+1}$. Concerning the second Hilbert coefficient $e_2(\mathbb{M})$, if $M$ is Cohen-Macaulay then
$$0 \leq e_2(\mathbb{M}) \leq  \underset{n \geq 1}{\sum} n\lambda(M_{n+1}/JM_n) \quad \quad (*)$$
(see \cite{HM}, \cite{N} and \cite{RV1}). If $M$ is no longer Cohen-Macaulay then new tools are necessary. The first Hilbert coefficient $e_1(\mathbb{M})$ was studied by Goto-Nishida in \cite{GN} and by Rossi-Valla in \cite{RV2}. We have the following inequalities
$$e_0(\mathbb{M}) - \lambda(M/M_1) \leq e_1(\mathbb{M}) - e_1(\mathbb{N}) \leq \underset{n \geq 0}{\sum} \lambda(M_{n+1}/JM_n)$$
where $\mathbb{N} = \{ J^nM\}$ is the $J$-adic filtration on $M$.

Little is known about $e_2(\mathbb{M})$. In \cite{LM} Mccune proved that if $\depth R \geq d-1$ and $\qq$ is a parameter ideal of $R$, then the second Hilbert coefficient of the $\qq$-adic filtration $\{\qq^n\}$ on $R$ is non-positive. In this paper we extend the inequalities $(*)$ in the case $M$ has almost maximal depth and we recover Mccune's result. More precisely, we prove the following main results.
\vskip 0.2cm
\noindent \textbf{Theorem 1.} \emph{Let $\mathbb{M} = \{ M_n \}$ be a good $\qq$-filtration of $R$-module $M$ of dimension $d \geq 2$ and $\depth M \geq d-1$. Suppose $J=(a_1,\cdots,a_d)$ is an ideal of $R$ generated by a maximal $\mathbb{M}$-superficial sequence for $\qq$. For each $i=1,\cdots,d-1$, denote the ideal $J_i = (a_1,\cdots,a_{d-i})$ of $R$. Then, we have
$$e_2(\mathbb{M}) \leq \underset{n \geq 1}{\sum} n  \lambda(M_{n+1}/JM_n).$$
Further, the equality holds if and only if $\depth gr_{\mathbb{M}}(M) \geq d-1$ and $(J_1M :_M a_d) \cap M_1 = J_1M$.}
\vskip 0.2cm
Example \ref{not almost} shows that the assumption on the depth of $M$ can not be weakened in the above result. Further if we consider $M = R$ and $\mathbb{N} = \{ J^n \}$ is the $J$-adic filtration of $R$, then our result implies the non-positivity of $e_2(\mathbb{N})$ which was proven by Mccune as above mentioned.

If $M$ is Cohen-Macaulay, then $e_2(\mathbb{N}) = 0$ and $gr_{\mathbb{N}}(M)$ is Cohen-Macaulay too. Under the assumption that the $gr_{\mathbb{N}}(M)$ has almost maximal depth, we may strengthen Theorem 1.
\vskip 0.2cm
\noindent \textbf{Theorem 2.} \emph{Let $\mathbb{M} = \{ M_n \}$ be a good $\qq$-filtration of $R$-module $M$ of dimension $d \geq 2$. Suppose $J$ is an ideal of $R$ generated by a maximal $\mathbb{M}$-superficial sequence for $\qq$ such that $\depth gr_{\mathbb{N}}(M) \geq d-1$, where $\mathbb{N}$ is the $J$-adic filtration. Then, we have
$$e_2(\mathbb{M}) - e_2(\mathbb{N}) \leq \underset{n \geq 1}{\sum} n  \lambda(M_{n+1}/JM_n).$$}

A lower bound for $e_2(\mathbb{M})$ is also given and it extends the result by Rees-Narita on the non-negativity of the second Hilbert coefficient proved in the Cohen-Macaulay case.
\vskip 0.2cm
\noindent \textbf{Theorem 3.} \emph{Let $\mathbb{M} = \{ M_n \}$ be a good $\qq$-filtration of $R$-module $M$ of dimension two and $\depth M > 0$. Suppose $J=(a_1,a_2)$ is an ideal of $R$ generated by a maximal $\mathbb{M}$-superficial sequence for $\qq$. Then, we have
$$e_2(\mathbb{M}) \geq -\binom{s+2}{2}\lambda\left(\frac{a_1M:a_2}{a_1M}\right),$$
where $s$ is the postulation number of the Ratliff-Rush filtration associated to $\mathbb{M}$.}

\section{Preliminaries}
Let $(R, \mm)$ be a Noetherian local ring with infinite residue field $R/\mm$ and $M$ a finitely generated $R$-module of dimension $d$. Let $\qq$ be an $\mm$-primary ideal of $R$ and $\mathbb{M} = \{ M_n \}$ a good $\qq$-filtration of $M$. Notice that if $N$ is a submodule of $M$ then $\mathbb{M}/N := \{ (M_n + N)/N \}$ is  a good $\qq$-filtration of the quotient module $M/N$.
\vskip 0.2cm
The Hilbert function of the associated graded module $gr_{\mathbb{M}}(M)$ is called the Hilbert function of filtration $\mathbb{M}$, by definition it is given by
$$H_{\mathbb{M}}(n) := \lambda(M_n/M_{n+1}) \ \text{for all} \ n \geq 0,$$
where $\lambda(-)$ denotes length. The Hilbert series of  filtration $\mathbb{M}$ is the power series
$$P_{\mathbb{M}}(t) := \underset{n \geq 0}{\sum} H_{\mathbb{M}}(n) t^n.$$

By Hilbert-Serre theorem, we can write
$$P_{\mathbb{M}}(t) = \frac{h_{\mathbb{M}}(t)}{(1-t)^d},$$
where $h_{\mathbb{M}}(t) = h_0(\mathbb{M}) + h_1(\mathbb{M})t + \cdots + h_s(\mathbb{M})t^s \in \Z[t]$ with $h_{\mathbb{M}}(1) \neq 0$, and $h_{\mathbb{M}}(t)$ is called the $h$-polynomial of $\mathbb{M}$. For every $i \geq 0$, we define
$$e_i(\mathbb{M}) := \frac{h_{\mathbb{M}}^{(i)}(1)}{i!},$$
and we notice that for $i = 0, 1, \cdots, d$ they coincide with the Hilbert coefficients already introduced. The polynomial
$$p_{\mathbb{M}}(X) = \sum_{i = 0}^{d-1} (-1)^i e_i(\mathbb{M})\binom{X+d -i - 1 }{d-i-1} \in \Q[X]$$
of degree $d-1$ is called the Hilbert polynomial of filtration $\mathbb{M}$ and the largest integer $n$ such that $H_{\mathbb{M}}(n)$ and $p_{\mathbb{M}}(n)$ disagree is called the \emph{postulation number} of filtration $\mathbb{M}$, denoted by $s(\mathbb{M})$.
\begin{Remark} \emph{By definition, we can write
$$e_i(\mathbb{M}) = \underset{k \geq i}{\sum} \binom{k}{i} h_k(\mathbb{M}),$$
where $h_k(\mathbb{M}) = 0$ for every $k \geq s$. Notice that $h_k(\mathbb{M})$ are not necessarily non-negative.}
\end{Remark}
The notion of superficial element plays a very important role in the study of Hilbert coefficients. We refer to \cite{RV1} for the definition and many interesting properties of superficial elements.  We only present here some basic properties that we use in this paper.
\begin{Proposition}\label{induction} \cite[Proposition 1.2]{RV1}.
Let $a \in \qq \setminus \qq^2$ be an $\mathbb{M}$-superficial element for $\qq$ and $d = \dim(M) \geq 1$. Then we have:
\vskip 0.2cm
\noindent {\rm (i)} $\dim(M/aM) = d-1$.
\vskip 0.2cm
\noindent {\rm (ii)} $e_i(\mathbb{M}/aM) = e_i(\mathbb{M})$  for $i = 0, 1, \cdots , d-2$.
\vskip 0.2cm
\noindent {\rm (iii)} $e_{d-1}(\mathbb{M}/aM) = e_{d-1}(\mathbb{M}) + (-1)^{d-1}\lambda(0 : a)$.
\vskip 0.2cm
\noindent {\rm (iv)} There exists an integer $j$ such that for every $n \geq j-1$ we have
$$e_d(\mathbb{M}/aM) = e_d(\mathbb{M}) + (-1)^d(\sum_{i = 0}^{n} \lambda(M_{i+1} : a/M_i) - (n+1)\lambda(0 : a)).$$
\noindent {\rm (v)} Denote $a^*$ the natural image of $a$ in $gr_{\qq}(R)$. Then $a^*$ is a regular element on $gr_{\mathbb{M}}(M)$ if and only if $a$ is $M$-regular and $e_d(\mathbb{M}/aM) = e_d(\mathbb{M})$.
\end{Proposition}
In the following proposition, we collect some important properties on the $\mathbb{M}$-superficial sequence for $\qq$ that are very useful to study the depth of $M$ and $gr_{\mathbb{M}}(M)$.
\begin{Proposition} \label{Prop depth} Let $a_1,\cdots,a_r$ be an $\mathbb{M}$-superficial sequence for $\qq$ and $I$ the ideal they generate. Then we have:
\vskip 0.2cm
\noindent {\rm (i)} $a_1,\cdots,a_r$ is a regular sequence on $M$ if and only if $\depth M \geq r$.
\vskip 0.2cm
\noindent {\rm (ii)} $a_1^*,\cdots,a_r^*$ is a regular sequence on $gr_{\mathbb{M}}(M)$ if and only if $\depth gr_{\mathbb{M}}(M) \geq r$.
\vskip 0.2cm
\noindent {\rm (iii)} (Valabrega-Valla) $a_1^*,\cdots,a_r^*$ is a regular sequence on $gr_{\mathbb{M}}(M)$ if and only if $a_1,\cdots,a_r$ is a regular sequence on $M$ and $IM \cap M_{n+1} = IM_n$ for every $n \geq 1$.
\vskip 0.2cm
\noindent {\rm (iv)} (Sally's machine) $\depth gr_{\mathbb{M}/IM}(M/IM) \geq 1$ if and only if $\depth gr_{\mathbb{M}}(M) \geq r+1.$
\end{Proposition}

Under our assumptions, always there exists an $\mathbb{M}$-superficial sequence $a_1, \cdots, a_d$ for $\qq$. In this case we say $a_1,\cdots,a_d$ is a maximal $\mathbb{M}$-superficial sequence for $\qq$.

In the end of this section we would like to recall some results in the one-dimensional case which is studied in  \cite[Sect. 2.2]{RV1}. By induction, these results are very useful to consider the higher dimension cases.

Let $M$ be an $R$-module of dimension one and $\mathbb{M} = \{ M_n \}$ a good $\qq$-filtration. For every $n \geq 0$ define
$$u_n(\mathbb{M}) := e_0(\mathbb{M}) - H_{\mathbb{M}}(n).$$
\vskip 0.2cm
\noindent Notice that $\dim M = 1$ implies $u_n(\mathbb{M}) = 0$ for $n \gg 0$. Moreover we have the following.
\begin{Lemma}\cite[Lemma 2.1]{RV1}\label{Form un}
Let $M$ be an $R$-module of dimension one. If $a$ is an $\mathbb{M}$-superficial element for $\qq$, then for every $n \geq 0$ we have
$$u_n(\mathbb{M}) = \lambda(M_{n+1}/aM_n) - \lambda(0 :_{M_n} a).$$
\end{Lemma}
The interesting point is that we can write down the Hilbert coefficients in terms of the integers $u_n(\mathbb{M})$.
\begin{Lemma}\cite[Lemma 2.2]{RV1}\label{Form e}
Let $M$ be an $R$-module of dimension one. Then for every $i \geq 0$ we have
$$e_i(\mathbb{M}) = \underset{n \geq i-1}{\sum} \binom{n}{i-1} u_n(\mathbb{M}).$$
\end{Lemma}

\section{The two-dimensional case}
In this section we consider the upper bound for the second Hilbert coefficient in the case $M$ is an $R$-module of dimension two and the depth of $M$ is positive. First we have the following result.
\begin{Theorem} \label{dim 2} Let $\mathbb{M} = \{ M_n \}$ be a good $\qq$-filtration of $R$-module $M$ of dimension two and $\depth M > 0$. Suppose $J=(a_1,a_2)$ is an ideal of $R$ generated by a maximal $\mathbb{M}$-superficial sequence for $\qq$. Then, we have
$$e_2(\mathbb{M}) \leq \underset{n \geq 1}{\sum} n  \lambda(M_{n+1}/JM_n).$$
Further, the equality holds if and only if $\depth gr_{\mathbb{M}}(M) > 0$ and $(a_1M :_M a_2) \cap M_1 = a_1M$.
\end{Theorem}
\begin{proof} Define $\overline{M} := M/a_1M$  and $\{\overline{M}_n\} := \mathbb{M}/a_1M = \{ (M_n+a_1M)/a_1M\}.$ Then $\{\overline{M}_n\}$ is a good $\qq$-filtration of $\overline{M}$ and $\dim(\overline{M}) = 1$. Since $\depth M > 0$, by Proposition \ref{Prop depth}, one has $a_1$ is $M$-regular. Hence, from Proposition \ref{induction} (iv), we have
\begin{equation}
e_2(\mathbb{M}/a_1M) = e_2(\mathbb{M}) + \sum_{i = 0}^{n} \lambda(M_{i+1} : a_1/M_i)  \ \text{for} \ n\gg0.
\end{equation}
Since $a_2$ is an $\mathbb{M}/a_1M$-superficial element for $\qq$, by Lemma \ref{Form un} for every $n \geq 1$
\begin{align*}
u_n(\mathbb{M}/a_1M) & = \lambda(\overline{M}_{n+1}/a_2\overline{M}_n) - \lambda(a_1M :_{\overline{M}_n} a_2) \\
                     & = \lambda\left(\frac{M_{n+1}}{JM_n + (a_1M \cap M_{n+1})}\right) - \lambda\left(\frac{(a_1M : a_2) \cap M_n + a_1M}{a_1M}\right) \\
                     & \leq \lambda(M_{n+1}/JM_n).
\end{align*}
Therefore, by Lemma \ref{Form e}, we have
\begin{equation}
e_2(\mathbb{M}/a_1M) = \underset{n \geq 1}{\sum} n u_n(\mathbb{M}/a_1M) \leq \underset{n \geq 1}{\sum} n  \lambda(M_{n+1}/JM_n).
\end{equation}
From (1) and (2) one has $e_2(\mathbb{M}) \leq e_2(\mathbb{M}/a_1M) \leq \underset{n \geq 1}{\sum} n  \lambda(M_{n+1}/JM_n).$
\vskip 0.2cm
If the equality holds then $e_2(\mathbb{M}) = e_2(\mathbb{M}/a_1M) = \underset{n \geq 1}{\sum} n  \lambda(M_{n+1}/JM_n).$ By Proposition \ref{induction} (v), $e_2(\mathbb{M}) = e_2(\mathbb{M}/a_1M)$ implies $\depth gr_{\mathbb{M}}(M) > 0$. Furthermore, from the inequality (2),
$$e_2(\mathbb{M}/a_1M) = \underset{n \geq 1}{\sum} n\lambda(M_{n+1}/JM_n) \Rightarrow u_n(\mathbb{M}/a_1M) =  \lambda(M_{n+1}/JM_n) \ \text{for every} \ n \geq 1.$$
In particular, $u_1(\mathbb{M}/a_1M) =  \lambda(M_{2}/JM_1)$ and this proves $(a_1M :_M a_2) \cap M_1 = a_1M$.

For the converse, if $\depth gr_{\mathbb{M}}(M) > 0$ then by Proposition \ref{Prop depth}, we have $a_1^*$ is a regular element on $gr_{\mathbb{M}}(M)$ and $a_1M \cap M_{n+1} = a_1M_n$ for every $n \geq 1$. On the other hand, for every $n \geq 1$
$$(a_1M : a_2) \cap M_n \subseteq (a_1M : a_2) \cap M_1 = a_1M.$$
This implies $(a_1M : a_2) \cap M_n + a_1M = a_1M$. Thus for every $n \geq 1$
\begin{align*}
u_n(\mathbb{M}/a_1M) &= \lambda\left(\frac{M_{n+1}}{JM_n + (a_1M \cap M_{n+1})}\right) - \lambda\left(\frac{(a_1M : a_2) \cap M_n + a_1M}{a_1M}\right)\\
                     &= \lambda(M_{n+1}/JM_n).
\end{align*}
Finally, since $\depth gr_{\mathbb{M}}(M) > 0$, we get
$$e_2(\mathbb{M}) = e_2(\mathbb{M}/a_1M) = \underset{n \geq 1}{\sum} n u_n(\mathbb{M}/a_1M) = \underset{n \geq 1}{\sum} n  \lambda(M_{n+1}/JM_n).$$
\end{proof}
Notice that if $M$ is Cohen-Macaulay then the condition $(a_1M :_M a_2) \cap M_1 = a_1M$ holds. However, the converse is not true. For instance, we see the following examples.
\begin{Example} \label{condition 1} \emph{Let $R = k[x,y,z,t]/((x^2,z^2)\cap(x-y,z+t))$. Then $R$ is a ring of dimension two and depth one. Let $J = (x^2+y^2,z^2+t^2)$ be a parameter ideal of $R$. Consider the good $J$-filtration $\mathbb{N} = \{ J^n \}$ of $R$. Then we have
$$((x^2+y^2) : z^2+t^2) \cap J = (x^2+y^2).$$
However, $R$ is not Cohen-Macaulay.}
\end{Example}
\begin{Example} \label{condition 2} \emph{(See \cite[Example 3.8]{LM}) Let $R = k[[x^5,xy^4,x^4y,y^5]] \cong k[[t_1,t_2,t_3,t_4]]/I$, where $I = (t_2t_3-t_1t_4,t_2^4-t_3t_4^3,t_1t_2^3-t_3^2t_4^2,t_1^2t_2^2-t_3^3t_4,t_1^3t_2-t_3^4,t_3^5-t_1^4t_4)$. Then $R$ is a domain of dimension two and depth one. Let $J = (x^5,y^5)$ be a parameter ideal of $R$. Consider the good $J$-filtration $\mathbb{N} = \{ J^n \}$ of $R$. Then $e_2(\mathbb{N}) = 0$, this means the equality in Theorem \ref{dim 2} holds; and we are able to check that
$$((x^5) : y^5) \cap J = (x^5).$$
However, $R$ is not Cohen-Macaulay.}
\end{Example}
Let $\mathbb{M} = \{ M_n \}$ be a good $\qq$-filtration of $R$-module $M$ and assume that $\depth M > 0$. Let $J=(a_1,a_2)$ be an ideal of $R$ generated by a maximal $\mathbb{M}$-superficial sequence for $\qq$. Denote by $\mathbb{N} := \{ J^nM\}$ the good $J$-adic filtration of $M$. Notice that, by \cite[Lemma 2.4]{RV1}, we can assume $a_1, a_2$ is also a maximal $\mathbb{N}$-superficial sequence for $J$. Denote by $\mathbb{N}/a_1M := \{(J^nM + a_1M)/a_1M\}$ the $(a_2)$-adic filtration of $\overline{M} = M/a_1M$. We have the following result.
\begin{Theorem} \label{difference dim 2} With the above assumptions we have
$$e_2(\mathbb{M}) - e_2(\mathbb{N}) \leq \underset{n \geq 1}{\sum} n  \lambda(M_{n+1}/JM_n) + \underset{n \geq 1}{\sum} \lambda(J^{n+1}M : a_1)/J^nM).$$
In particular, if $\depth gr_{\mathbb{N}}(M) > 0$ then
$$e_2(\mathbb{M}) - e_2(\mathbb{N}) \leq \underset{n \geq 1}{\sum} n  \lambda(M_{n+1}/JM_n).$$
\end{Theorem}
\begin{proof}
Since $a_2$ is a $\mathbb{N}/a_1M$-superficial element for $J$, by Lemma \ref{Form un} for every $n \geq 1$
\begin{align*}
u_n(\mathbb{N}/a_1M) & = \lambda\left(\frac{J^{n+1}M}{J^{n+1}M + (a_1M \cap J^{n+1}M)}\right) - \lambda\left(\frac{(a_1M : a_2) \cap J^nM + a_1M}{a_1M}\right) \\
                     & = - \lambda\left(\frac{(a_1M : a_2) \cap J^nM + a_1M}{a_1M}\right).
\end{align*}
Since $e_2(\mathbb{M}) \leq e_2(\mathbb{M}/a_1M)$, we have
\begin{align*}
e_2(\mathbb{M}) - e_2(\mathbb{N}/a_1M) & \leq e_2(\mathbb{M}/a_1M) - e_2(\mathbb{N}/a_1M) \\
                                       & = \underset{n \geq 1}{\sum} n(u_n(\mathbb{M}/a_1M) - u_n(\mathbb{N}/a_1M)) \\
                                       & \leq  \underset{n \geq 1}{\sum} n  \lambda(M_{n+1}/JM_n).
\end{align*}
On the other hand, by Proposition \ref{induction} (iv) we have
$$e_2(\mathbb{N}/a_1M) = e_2(\mathbb{N}) + \underset{n \geq 1}{\sum} \lambda(J^{n+1}M : a_1)/J^nM).$$
Thus $e_2(\mathbb{M}) - e_2(\mathbb{N}) \leq \underset{n \geq 1}{\sum} n  \lambda(M_{n+1}/JM_n) + \underset{n \geq 1}{\sum} \lambda(J^{n+1}M : a_1)/J^nM).$

In particular, if $\depth gr_{\mathbb{N}}(M) > 0$ then $a_1^*$ is $gr_{\mathbb{N}}(M)$-regular and $(J^{n+1}M : a_1) = J^nM$ for every $n \geq 1$. Hence, $$e_2(\mathbb{M}) - e_2(\mathbb{N}) \leq \underset{n \geq 1}{\sum} n  \lambda(M_{n+1}/JM_n).$$
\end{proof}
\begin{Remark}\emph{In the proof of Theorem \ref{difference dim 2}, one observes that without assumption on the depth of $gr_{\mathbb{N}}(M)$ we have the following bound
$$e_2(\mathbb{M}) - e_2(\mathbb{N}/a_1M) \leq  \underset{n \geq 1}{\sum} n  \lambda(M_{n+1}/JM_n).$$}
\end{Remark}
We now consider the lower bound for the second Hilbert coefficient. In the case $(R,\mm)$ is a Cohen-Macaulay local ring, Narita proved in \cite{N} that $e_2(\qq) \geq 0$ for every $\mm$-primary ideal $\qq$ of $R$, where $e_2(\qq)$ is the second Hilbert coefficient of the $\qq$-adic filtration of $R$. The non-negativity of the second Hilbert coefficient was extended to the case of Cohen-Macaulay modules, see for instance \cite[Proposition 3.1]{RV1}. However, in the case the module $M$ is not Cohen-Macaulay then the second Hilbert coefficient could be negative. For instance, we see the following example.
\begin{Example} \emph{Let $R=k[[x,y,z]]/(x^2,xy)$ be a local ring of dimension $2$ and $\depth R = 1$. Then the Hilbert series of the $\mm$-adic filtration of $R$ is the following
$$P_{\mm}(t) = \frac{1+t-t^2}{(1-t)^2}.$$
This means $e_2(\mm) = -1 <0$.}
\end{Example}
Rossi and Valla in \cite{RV1} used a very effective device to study the Hilbert coefficients that is the Ratliff-Rush filtration. Let $\qq$ be an $\mm$-primary ideal of $R$ and $\mathbb{M} = \{M_n\}$ a good $\qq$-filtration on the module $M$. For every $n \geq 0$ we define
$$\widetilde{M}_n := \underset{k \geq 1}{\bigcup} (M_{n+k} :_M \qq^k).$$
Notice that since $M$ is Noetherian, there exists a positive integer $t$ (depending on $n$) such that
$$\widetilde{M}_n = M_{n+k} : \qq^k \quad \forall k \geq t.$$
Then $\widetilde{\mathbb{M}} := \{ \widetilde{M}_n\}$ is a good $\qq$-filtration on $M$ and it is called the Ratliff-Rush filtration associated to $\mathbb{M}$. We refer to \cite{PZ}, \cite{RR} and \cite[Sect. 3.1]{RV1} for more properties of the Ratliff-Rush filtration. By using this approach we give a lower bound for the second Hilbert coefficient in the case $M$ has almost maximal depth. More precisely, we have the following result involving the postulation number of $\widetilde{\mathbb{M}}$ (see the definition of postulation number in section 2).
\begin{Theorem}  \label{lower bound} Let $\mathbb{M} = \{ M_n \}$ be a good $\qq$-filtration of $R$-module $M$ of dimension two and $\depth M > 0$. Suppose $J=(a_1,a_2)$ is an ideal of $R$ generated by a maximal $\mathbb{M}$-superficial sequence for $\qq$. Then, we have
$$e_2(\mathbb{M}) \geq -\binom{s+2}{2}\lambda\left(\frac{a_1M:a_2}{a_1M}\right),$$
where $s$ is the postulation number of the Ratliff-Rush filtration associated to $\mathbb{M}$.
\end{Theorem}
\begin{proof} Let $\widetilde{\mathbb{M}} = \{ \widetilde{M}_n\}$ be the Ratliff-Rush filtration associated to $\mathbb{M}$. By \cite[Lemma 3.1]{RV1} we have $\depth gr_{\widetilde{\mathbb{M}}}(M) \geq 1$. Hence, by Proposition \ref{induction} (v) and Lemma \ref{Form e}
$$e_2(\widetilde{\mathbb{M}}) = e_2(\widetilde{\mathbb{M}}/a_1M) = \underset{n \geq 1}{\sum} n u_n(\widetilde{\mathbb{M}}/a_1M).$$
By \cite[Lemma 3.1]{RV1}, $a_1, a_2$ is also a maximal $\widetilde{\mathbb{M}}$-superficial sequence for $\qq$, so that $a_2$ is a $\widetilde{\mathbb{M}}/a_1M$-superficial element for $\qq$. Hence, by Lemma \ref{Form un} for every $n \geq 1$
\begin{align*}
u_n(\widetilde{\mathbb{M}}/a_1M) &= \lambda(\widetilde{M}_{n+1}/J\widetilde{M}_n) - \lambda\left(\frac{(a_1M:a_2) \cap \widetilde{M}_n+a_1M}{a_1M}\right) \\
                                 &\geq - \lambda\left(\frac{a_1M:a_2}{a_1M}\right)
\end{align*}
Since $\depth gr_{\widetilde{\mathbb{M}}}(M) \geq 1$, once has $s(\widetilde{\mathbb{M}}/a_1M) = s(\widetilde{\mathbb{M}})+1$. Hence $u_n(\widetilde{\mathbb{M}}/a_1M) = 0$ for every $n \geq s+2$. Thus, by  \cite[Lemma 3.1]{RV1} we have
$$e_2(\mathbb{M}) = e_2(\widetilde{\mathbb{M}}) = \sum_{n = 1}^{s+1} n u_n(\widetilde{\mathbb{M}}/a_1M) \geq - \sum_{n = 1}^{s+1} n\lambda\left(\frac{a_1M:a_2}{a_1M}\right) = -\binom{s+2}{2}\lambda\left(\frac{a_1M:a_2}{a_1M}\right).$$
\end{proof}
If $M$ is Cohen-Macaulay then $(a_1M:a_2) = a_1M$. Moreover, because the study of $e_2(\mathbb{M})$ can be reduced to the $2$-dimensional modules by Proposition \ref{induction}, the above theorem implies the non-negativity of the second Hilbert coefficient in the Cohen-Macaulay modules.
\begin{Corollary}\emph{(see \cite{N})} Let $\mathbb{M}$ be a good $\qq$-filtration of the Cohen-Macaulay module $M$ of dimension $d \geq 2$. Then
$$e_2(\mathbb{M}) \geq 0.$$
\end{Corollary}

\section{The higher dimensional case}
In this section we extend Theorem \ref{dim 2}, Theorem \ref{difference dim 2} and Theorem \ref{lower bound} to the higher dimensions. The following result extends Theorem \ref{dim 2}.
\begin{Theorem}\label{higher dim}
Let $\mathbb{M} = \{ M_n \}$ be a good $\qq$-filtration of $R$-module $M$ of dimension $d \geq 2$ and $\depth M \geq d-1$. Suppose $J=(a_1,\cdots,a_d)$ is an ideal of $R$ generated by a maximal $\mathbb{M}$-superficial sequence for $\qq$. For each $i=1,\cdots,d-1$, denote the ideal $J_i = (a_1,\cdots,a_{d-i})$ of $R$. Then, we have
$$e_2(\mathbb{M}) \leq \underset{n \geq 1}{\sum} n  \lambda(M_{n+1}/JM_n).$$
Further, the equality holds if and only if $\depth gr_{\mathbb{M}}(M) \geq d-1$ and $(J_1M :_M a_d) \cap M_1 = J_1M$.
\end{Theorem}
\begin{proof} By Theorem \ref{dim 2} it is enough to consider the case $d \geq 3$. Define $\overline{M} := M/J_2M$  and $\{\overline{M}_n\} := \mathbb{M}/J_2M = \{ (M_n+J_2M)/J_2M\}.$ Then $\{\overline{M}_n\}$ is a good $\qq$-filtration of $\overline{M}$ and $\dim(\overline{M}) = 2$.

Since $\depth(M) \geq d-1$, one has $J_2$ is generated by a regular sequence and $\depth(\overline{M}) = \depth(M) - (d-2) \geq 1$. Hence, by Proposition \ref{induction}
 $$e_2(\mathbb{M}) = e_2(\mathbb{M}/J_{d-1}M) = \cdots = e_2(\mathbb{M}/J_2M).$$
Denote by $K = (a_{d-1},a_d)$ the ideal generated by a maximal $\mathbb{M}/J_2M$-superficial sequence for $\qq$, then by Theorem \ref{dim 2} we have
\begin{align*}
e_2(\mathbb{M}/J_2M) & \leq \underset{n \geq 1}{\sum} n  \lambda(\overline{M}_{n+1}/K\overline{M}_n) \\
&= \underset{n \geq 1}{\sum} n \lambda\left(\frac{M_{n+1}+J_2M}{KM_n + J_2M}\right) \\
&= \underset{n \geq 1}{\sum} n \lambda\left(\frac{M_{n+1}+J_2M}{JM_n + J_2M}\right) \\
&= \underset{n \geq 1}{\sum} n \lambda\left(\frac{M_{n+1}}{JM_n + (J_2M \cap M_{n+1}) }\right) \\
&\leq \underset{n \geq 1}{\sum} n \lambda(M_{n+1}/JM_n)
\end{align*}
Thus $e_2(\mathbb{M}) \leq \underset{n \geq 1}{\sum} n \lambda(M_{n+1}/JM_n)$.

If the equality holds then $e_2(\mathbb{M}/J_2M) = \underset{n \geq 1}{\sum} n  \lambda(\overline{M}_{n+1}/K\overline{M}_n)$. By Theorem \ref{dim 2} we have $\depth gr_{\mathbb{M}/J_2M}(M/J_2M) \geq 1$ and $(a_{d-1}\overline{M} :_{\overline{M}} a_d) \cap \overline{M}_1 = a_{d-1}\overline{M}.$ Hence, by Proposition \ref{Prop depth} (iv), we have $\depth gr_{\mathbb{M}}(M) \geq d-1$. Moreover,
\begin{align*}
& \ (a_{d-1}\overline{M} :_{\overline{M}} a_d) \cap \overline{M}_1 = a_{d-1}\overline{M}\\
\Leftrightarrow & \ (J_1M/J_2M :_{\overline{M}} a_d) \cap (M_1 + J_2M)/J_2M =  J_1M/J_2M\\
\Leftrightarrow & \ (J_1M :_M a_d) \cap M_1 = J_1M.
\end{align*}

For the converse, if $\depth gr_{\mathbb{M}}(M) \geq d-1$ then $\depth gr_{\mathbb{M}/J_2M}(M/J_2M) \geq 1$. Hence, by Theorem \ref{dim 2}, we have
$$e_2(\mathbb{M}/J_2M) = \underset{n \geq 1}{\sum} n  \lambda(\overline{M}_{n+1}/K\overline{M}_n) = \underset{n \geq 1}{\sum} n \lambda\left(\frac{M_{n+1}}{JM_n + (J_2M \cap M_{n+1}) }\right).$$
Since $\depth gr_{\mathbb{M}}(M) \geq d-1$, by Proposition \ref{Prop depth} (iii), we have
$$J_2M \cap M_{n+1} = J_2M_n \subseteq JM_n, \forall n \geq 1.$$
Thus $e_2(\mathbb{M}) = e_2(\mathbb{M}/J_2M) = \underset{n \geq 1}{\sum} n  \lambda(M_{n+1}/JM_n).$
\end{proof}
Notice that in Theorem \ref{higher dim} the condition $\depth M \geq d-1$ is necessary as the following example shows.
\begin{Example}\label{not almost} \emph{\cite[Example 3.7]{LM}. Let $R = k[x,y,z,u,v,w]/I$ where $I$ is the intersection of ideals $I = (x+y,z-u,w) \cap (z,u-v,y) \cap (x,u,w)$. Then $R$ is a ring of dimension three and depth one. Let $\qq = (u-y,z+w,x-v)$ be a parameter ideal of $R$. Consider the $\qq$-adic filtration $\mathbb{N} = \{\qq^n\}$ of $R$. We have $e_2(\mathbb{N}) = 1 > 0$ and this means that the bound for $e_2$ in Theorem \ref{higher dim} is not satisfied.}
\end{Example}
Let $(R,\mm)$ be a local ring and $\qq$ an $\mm$-primary ideal of $R$. For $i\geq0$, denote by $e_i(\qq)$ the Hilbert coefficient of the $\qq$-adic filtration $\{\qq^n\}$ of $R$. Theorem \ref{higher dim} implies the result by Mccune on non-positivity of the second Hilbert coefficient for the parameter ideals (See \cite[Theorem 3.5]{LM}).
\begin{Corollary} Let $(R, \mm)$ be a local ring of dimension $d \geq 2$  and $\depth R \geq d-1$.  Let $\qq \subseteq R$ be a parameter ideal. Then, we have $e_2(\qq) \leq 0.$
\end{Corollary}
As a consequence of Theorem \ref{higher dim} in case $M$ is Cohen-Macaulay we get the bound for the second Hilbert coefficient given by Rossi and Valla in \cite[Theorem 2.5]{RV1}.
\begin{Corollary} Let $\mathbb{M} = \{ M_n \}$ be a good $\qq$-filtration of a Cohen-Macaulay module $M$ of dimension $d \geq 2$.  Suppose $J$ is an ideal of $R$ generated by a maximal $\mathbb{M}$-superficial sequence for $\qq$. Then, we have
$$e_2(\mathbb{M}) \leq \underset{n \geq 1}{\sum} n  \lambda(M_{n+1}/JM_n).$$
Further, the equality holds if and only if $\depth gr_{\mathbb{M}}(M) \geq d-1$.
\end{Corollary}
Applying Theorem \ref{higher dim} for the $\mm$-adic filtration of $R$ we have the following stronger result which extends the previous result in the particular case $M = R$ and $\qq = \mm$.
\begin{Corollary} Let $(R, \mm)$ be a local ring of dimension $d \geq 2$  and $\depth R \geq d-1$.  Suppose $J=(x_1,\cdots,x_d)$ is an ideal of $R$ generated by a maximal $\mm$-adic superficial sequence. Then, we have
$$e_2(\mm) \leq \underset{n \geq 1}{\sum} n  \lambda(\mm^{n+1}/J\mm^n).$$
Further, the equality holds if and only if $R$ is Cohen-Macaulay and $\depth gr_{\mm}(R) \geq d-1$.
\end{Corollary}

We extend Theorem \ref{difference dim 2} to the higher dimensions.
\begin{Theorem}\label{bound for difference}
Let $\mathbb{M} = \{ M_n \}$ be a good $\qq$-filtration of $R$-module $M$ of dimension $d \geq 2$. Suppose $J$ is an ideal of $R$ generated by a maximal $\mathbb{M}$-superficial sequence for $\qq$ such that $\depth gr_{\mathbb{N}}(M) \geq d-1$, where $\mathbb{N}$ is the $J$-adic filtration. Then, we have
$$e_2(\mathbb{M}) - e_2(\mathbb{N}) \leq \underset{n \geq 1}{\sum} n  \lambda(M_{n+1}/JM_n).$$
\end{Theorem}
\begin{proof} We proceed by induction on $d$. The case $d=2$ is proved by Theorem \ref{difference dim 2}. Let $d \geq 3$. By \cite[Lemma 2.4]{RV1} we can assume $J = (a_1,\cdots,a_d)$ where $a_1,\cdots,a_d$ is a maximal sequence of superficial elements for $J$ with respect to $\mathbb{M}$ and $\mathbb{N}$. By Proposition \ref{induction} we have
$$e_2(\mathbb{M}/a_1M) - e_2(\mathbb{M}) = e_2(\mathbb{N}/a_1M) - e_2(\mathbb{N}) =\left\{
                              \begin{array}{ll}
                                \lambda(0 : a_1) & \hbox{if $d = 3$;} \\
                                0 & \hbox{if $d > 3$.}
                              \end{array}
                            \right.$$
Let $K = (a_2,\cdots,a_d)$, then $K$ is generated by a maximal $\mathbb{M}/a_1M$-superficial sequence and $\mathbb{N}/a_1M$ is the $K$-adic filtration on $M/a_1M$. By Proposition \ref{Prop depth}, $\depth gr_{\mathbb{N}}(M) \geq d-1 \geq 2$ implies $\depth gr_{\mathbb{N}/a_1M}(M/a_1M) \geq d-2$. Thus by induction we get
\begin{align*}
e_2(\mathbb{M}) - e_2(\mathbb{N}) & = e_2(\mathbb{M}/a_1M) - e_2(\mathbb{N}/a_1M) \\
&\leq \underset{n \geq 1}{\sum} n \lambda\left(\frac{M_{n+1}+a_1M}{KM_n + a_1M}\right) \\
&= \underset{n \geq 1}{\sum} n \lambda\left(\frac{M_{n+1}+a_1M}{JM_n + a_1M}\right) \\
&= \underset{n \geq 1}{\sum} n \lambda\left(\frac{M_{n+1}}{JM_n + (a_1M \cap M_{n+1}) }\right) \\
&\leq \underset{n \geq 1}{\sum} n \lambda(M_{n+1}/JM_n)
\end{align*}
\end{proof}
The following result is the generalization of Theorem \ref{lower bound}.
\begin{Proposition} Let $\mathbb{M}=\{M_n\}$ be a good $\qq$-filtration of $R$-module $M$ of dimension $d\geq2$ and $\depth M \geq d-1$. Suppose $a_1,a_2,\cdots,a_d$ is a maximal $\mathbb{M}$-superficial sequence for $\qq$. Then
$$e_2(\mathbb{M})\geq -\binom{s+2}{2}\lambda\left(\frac{(a_1,\cdots,a_{d-1})M:a_d}{(a_1,\cdots,a_{d-1})M}\right).$$
where $s$ is the postulation number of the Ratliff-Rush filtration associated to $\mathbb{M}/(a_1,\cdots,a_{d-2})M$.
\end{Proposition}
\begin{proof} Define $\overline{M} = M/(a_1,\cdots,a_{d-2})M$. Then
$$\overline{\mathbb{M}} = \mathbb{M}/(a_1,\cdots,a_{d-2})M = \left\{\frac{M_n + (a_1,\cdots,a_{d-2})M}{(a_1,\cdots,a_{d-2})M}\right\} $$
is a good $\qq$ filtration of the $2$-dimensional module $\overline{M}$. Since $K = (a_{d-1},a_d)$ is an ideal of $R$ generated by a maximal $\overline{\mathbb{M}}$-superficial sequence, by Theorem \ref{lower bound}, we get
$$e_2(\mathbb{M}) = e_2(\overline{\mathbb{M}}) \geq -\binom{s+2}{2}\lambda\left(\frac{a_{d-1}\overline{M}:_{\overline{M}} a_d}{a_{d-1}\overline{M}}\right),$$
where $s$ is a postulation number of the Ratliff-Rush filtration associated to $\overline{\mathbb{M}}$. Finally, the conclusion follows by the fact
$$ \frac{(a_{d-1}\overline{M}:_{\overline{M}} a_d)}{a_{d-1}\overline{M}} \cong  \frac{((a_1,\cdots,a_{d-1})M : a_d)}{(a_1,\cdots,a_{d-1})M}.$$
\end{proof}

\noindent \textbf{Acknowledgements}
I thank my advisor Maria Evelina Rossi for suggesting the problem and for providing helpful suggestions throughout the preparation of this manuscript. I am also grateful to the department of Mathematics of Genova University for supporting my PhD program.

\end{document}